\documentclass{amsart}
\usepackage[british]{babel}
\usepackage{amsmath,amssymb,amsthm,amscd}
\usepackage{graphicx,subfigure}
\usepackage{color}
\usepackage{pb-diagram}
\usepackage{tabularx}
\usepackage{marginnote}
\usepackage{verbatim}
\usepackage{comment}

\newtheorem{theorem}{Theorem}[section]
\newtheorem{proposition}[theorem]{Proposition}
\newtheorem{corollary}[theorem]{Corollary}

\newtheorem{lemma}[theorem]{Lemma}

\theoremstyle{definition}
\newtheorem{definition}[theorem]{Definition}

\theoremstyle{remark}
\newtheorem{remark}[theorem]{Remark}

\numberwithin{equation}{section}

\def\AA{{\mathbb A}}
\newcommand{\RR}{{\mathbb R}}

\newcommand{\TT}{{\mathbb T}}
\newcommand{\ZZ}{{\mathbb Z}}

\newcommand{\norm}[1]{\|{#1}\|}

\def\epsilon{\varepsilon}

\def\Dif{ {\mbox{\rm D}} }

\def\df{\dot{f}}       % derivative of f
\def\ddf{\ddot{f}}     % 2nd derivative of f

\def\to{t}        % Old t
\def\Pto{\bar t}  % New t
\def\vo{v}        % Old v
\def\Pvo{\bar v}  % New v
\def\Eo{e}        % Old E
\def\PEo{\bar e}  % New E
\allowdisplaybreaks

\begin{document}

\title{Some remarks on the periodic motions of a bouncing ball}

\author{Stefano Mar\`o}
\address{Universidad de Oviedo, Departamento de Matem\'aticas, Calle Leopoldo Calvos Sotelo 18, 33007 Oviedo, Spain}
%\curraddr{}
\email{marostefano@uniovi.es}

\thanks{This work has been supported by the project PID2021-128418NA-I00 awarded by the Spanish Ministry of Science and Innovation. The work is also part of the author's activity within the DinAmicI community (www.dinamici.org) and the Gruppo Nazionale di Fisica Matematica, INdAM}

\begin{abstract}
We consider the vertical motion of a free falling ball bouncing elastically on a racket moving in the vertical direction according to a regular $1$-periodic function $f$. For fixed coprime $p,q$ we study existence, stability in the sense of Lyapunov and multiplicity of $p$ periodic motions making $q$ bounces in a period. If $f$ is real analytic we prove that one periodic motion is unstable and give some information on the set of these motions.
\end{abstract}
\maketitle

%\tableofcontents

\section{Introduction}
\noindent The vertical dynamics of a free falling ball on a moving racket is considered. The racket is supposed to
move periodically in the vertical direction according to a regular periodic function $f(t)$ and
the ball is reflected according to the law of elastic bouncing when hitting the
racket. The only force acting on the ball is the gravity $g$. Moreover, the mass of the racket is assumed to be large with respect to the mass of the
ball so that the impacts do not affect the motion of the racket.

\noindent This model has inspired many authors as it represents a simple model exhibiting complex dynamics, depending on the properties of the function $f$. The first results were given by Pustyl'nikov in \cite{pust} who studied the possibility of having motions with velocity tending to infinity, for $\df$ large enough. On the other hand, KAM theory implies that if the $C^k$ norm, for $k$ large, of $\df$ is small then all motions are bounded. On these lines some recent results are given in \cite{ma_xu,maro4,maro6}. Bounded motions can be regular (periodic and quasiperiodic, see \cite{maro3}) and chaotic (see \cite{maro5,maro2,ruiztorres}). Moreover, the non periodic case is studied in \cite{kunzeortega2,ortegakunze}, the case of different potentials is considered in \cite{dolgo} and recent results on ergodic properties are present in \cite{studolgo}.      

In this paper we are concerned with $(p,q)$-periodic motions understood as $p$-periodic with $q$ bounces in each period. Here $p,q$ are supposed to be positive coprime integers. In \cite{maro3} it is proved that if $p/q$ is sufficiently large, then there exists at least one $(p,q)$-periodic motion. This result comes from an application of Aubry-Mather theory as presented in \cite{bangert}. Actually, the bouncing motions correspond to the orbits of an exact symplectic twist map of the cylinder. The orbits of such maps can be found as critical points of an action functional and the $(p,q)$-periodic orbits found in \cite{maro3} correspond to minima. Here we first note that a refined version of Aubry-Mather theory (see \cite{katok_hass}) gives, for each couple of coprime $p,q$ such that $p/q$ is large, the existence of another $(p,q)$-periodic orbit that is not a minimum, since it is found via a minimax argument. This gives the existence of two different $(p,q)$-periodic motions for fixed values of $p,q$, with $p/q$ large.

We are first interested in the stability in the sense of Lyapunov of such periodic motions. This is related to the structure of the $(p,q)$-periodic orbits of the corresponding exact symplectic twist map, for fixed $p,q$. It comes from Aubry-Mather theory that the $(p,q)$-periodic orbits that are minima can be ordered and if there are two with no other in the middle then they are connected by heteroclinic orbits. In this case they are unstable. On the other hand, $(p,q)$-periodic orbits can form an invariant curve. In this case they are all minima but their stability cannot be determined as before since we are in a degenerate scenario. However, if we suppose to be in the real analytic case, a topological argument (see \cite{ortega_fp,ortega_book}) can be used to deduce instability. More precisely we will use the fact that for a real analytic area and orientation preserving embedding of the plane that is not the identity, every stable fixed point is isolated. Therefore, the hypothesis of $f$ being real analytic comes here into play.

Concerning the structure of the set of $(p,q)$-periodic motions we prove that in the real analytic case they can only be either isolated or in a degenerate case, in the sense that the corresponding orbits form an invariant curve that is the graph of a real analytic function. As before, in the isolated case at least one is unstable and in the degenerate case they all are minima and unstable. Note that this result differs from Aubry-Mather theory since we are not requiring the orbits to be minima of a functional. To prove this result we need the $q$-th iterate of the map to be twist. For $q=1$ this is true for every real analytic $f$, while for the general case $q>1$ we need to restrict to $\norm{\ddf}$ being small.

The paper is organized as follows. In Section \ref{sec:theory} we recall some known facts about exact symplectic twist maps together with the results for the analytic case. In Section \ref{sec:tennis} we introduce the bouncing ball map and describe its main properties. Finally, the results on the existence of two $(p,q)$-periodic motions, the instability and the structure of the set are given in Section \ref{sec:per}.

\section{Some results on periodic orbits of exact symplectic twist maps}\label{sec:theory}
Let us denote by $\Sigma=\RR\times(a,b)$ with $-\infty\leq a<b\leq +\infty$ a possibly unbounded strip of $\RR^2$. We will deal with $C^k$ ($k\geq 1$) or real analytic embeddings $\tilde{S}:\Sigma \rightarrow \RR^2$ such that
\begin{equation}\label{def_cyl}
\tilde{S}\circ\sigma=\sigma\circ \tilde{S}
\end{equation}
where $\sigma:\RR^2\rightarrow\RR^2$ and $\sigma(x,y)=(x+1,y)$. By this latter property, $\tilde{S}$ can be seen as the lift of an embedding $S:\Sigma\rightarrow\AA$ where $\AA = \TT\times \RR$ with $\TT = \RR/\ZZ$ and $\Sigma$ is now understood as the corresponding strip of the cylinder. We denote $\tilde{S}(x,y)=(\bar{x},\bar{y})$ and the corresponding orbit by $(x_n,y_n)_{n\in\ZZ}$. 

 We say that $\tilde{S}$ is exact symplectic if there exists a $C^1$ function $V:\Sigma\rightarrow \RR$ such that $V\circ\sigma=V$ and
\[
\bar{y} d \bar{x} -y dx = dV(x,y) \quad\mbox{in }\Sigma.
\]
Moreover, by the (positive) twist condition we understand
\[
\frac{\partial \bar{x}}{\partial y} >0 \quad\mbox{in }\Sigma. 
\]
A negative twist condition would give analogous results. The exact symplectic condition implies that $\tilde{S}$ preserves the two-form $dy\wedge dx$ so that it is area and orientation preserving.
An equivalent characterization is the existence of a generating function, i.e. a $C^2$ function $h:\Omega\subset \RR^2\rightarrow \RR$ such that $h(x+1,\bar{x}+1) = h(x,\bar{x})$ and $h_{12}(x,\bar{x}) <0$ in $\Omega$ and for $(x,y)\in\Sigma$ we have $\tilde{S}(x,y) = (\bar{x},\bar{y})$ if and only if 
    \[
    \left\{
    \begin{split}
      h_1(x,\bar{x})&=-y \\
      h_2(x,\bar{x})&=\bar{y}.
    \end{split}
    \right.
    \]
Moreover, $\tilde{S}$ preserves the ends of the cylinder if, uniformly in $x$,
\[
\bar{y}(x,y) \rightarrow \pm\infty \quad\mbox{as }y\rightarrow \pm\infty
\]
and twists each end infinitely if, uniformly in $x$,
\[
\bar{x}(x,y)-x \rightarrow \pm\infty \quad\mbox{as }y\rightarrow \pm\infty.
\]
%
%Finally we say that two embeddings $\tilde{S}_1,\tilde{S}_2:\Sigma\rightarrow\RR^2$ satisfying condition \eqref{def_cyl} are isotopic if there exists a continuous function $H:\Sigma^2\times[0,1]\rightarrow\RR^2$ such that $H(\sigma(x,y),\lambda)=\sigma(H(x,y,\lambda))$ for every $\lambda\in[0,1]$ and $(x,y)\in\Sigma$, $H(\cdot,0)=\tilde{S}_1$, $H(\cdot,1)=\tilde{S}_2$ and for every $\lambda\in[0,1]$ the functions $H(\cdot,\lambda)$ are homeomorphisms.
Finally we will say that an embedding of the cylinder $S:\Sigma\rightarrow\AA$ satisfies any of these properties if so does its corresponding lift.

These maps enjoy several properties and many interesting orbits are proved to exist. Here we recall some results concerning periodic orbits. We start with the following
\begin{definition}
  Fix two coprime integers $p,q$ with $q\neq 0$. An orbit $(x_n,y_n)_{n\in\ZZ}$ of $\tilde{S}$ is said $(p,q)$-periodic if $(x_{n+q},y_{n+q})=(x_n+p,y_n)$ for every $n\in\ZZ$. Moreover, we say that it is stable (in the sense of Lyapunov) if for every $\varepsilon>0$ there exists $\delta>0$ such that for every $(\hat{x}_0,\hat{y}_0)$ satisfying $|(x_0,y_0)-(\hat{x}_0,\hat{y}_0)|<\delta$ we have $|\tilde{S}^n(\hat{x}_0,\hat{y}_0)-({x}_n,{y}_n)|<\varepsilon$ for every $n\in\ZZ$. 
\end{definition}

\begin{remark}\label{rem_periodic}
  Note that $(q,p)$-periodic orbits correspond to fixed points of the map $\sigma^{-p}\circ \tilde{S}^q$. This follows from the fact that $\tilde{S}$ is a diffeomorphism defined on the cylinder. Each point of the orbit is a fixed point of $\sigma^{-p}\circ\tilde{S}^q$ and a fixed point of $\sigma^{-p}\circ\tilde{S}^q$ is the initial condition of a $(p,q)$-periodic orbit. Note that different fixed points may correspond to the same orbit but not viceversa. Moreover, if an orbit is $(q,p)$-periodic then it cannot be also $(q',p')$-periodic unless $p/q=p'/q'$. Actually, let $\xi=(x,y)$ and suppose that $\xi=\sigma^{-p}\circ\Phi^q(\xi)=\sigma^{-p'}\circ \Phi^{q'}(\xi)$. Then $\sigma^{p}(\xi)=\Phi^q(\xi)$ and $\sigma^{p'}(\xi)=\Phi^{q'}(\xi)$ from which $\Phi^{qq'}(\xi)=\sigma^{pq'}(\xi)=\sigma^{p'q}(\xi)$ and $pq'=p'q$.   
  Finally, the stability of a $(q,p)$-periodic orbit corresponds to the stability in the sense of Lyapunov of the corresponding fixed point of the map $\sigma^{-p}\circ\tilde{S}^q$.
\end{remark}

A particular class of periodic orbits are the so called Birkhoff periodic orbits.
\begin{definition}
  Fix two coprime integers $p,q$ with $q\neq 0$. An orbit $(x_n,y_n)_{n\in\ZZ}$ of $\tilde{S}$ is said a Birkhoff $(p,q)$-periodic orbit if there exists a sequence $(s_n,u_n)_{n\in\ZZ}$ such that
  \begin{itemize}
  \item $(s_0,u_0)=(x_0,y_0)$ 
  \item $s_{n+1}>s_n$ 
  \item $(s_{n+q},u_{n+q})=(s_{n}+1,u_{n})$
  \item $(s_{n+p},u_{n+p})=\tilde{S}(s_{n},u_{n})$
    \end{itemize}
 \end{definition}
  \begin{remark}
    Note that a Birkhoff $(p,q)$-periodic orbit is a $(p,q)$-periodic orbit since
    \[
(x_{n+q},y_{n+q})=(s_{np+qp},u_{np+qp})=(s_{np}+p,u_{np})=\tilde{S}^n(s_0,u_0)+(p,0)=(x_{n}+p,y_{n}).
    \]
    \end{remark}

The existence of Birkhoff $(p,q)$-periodic orbits comes from Aubry-Mather theory. Here we give a related result \cite[Th. 9.3.7]{katok_hass}

\begin{theorem}\label{birk_orbits}
  Let $S:\AA\rightarrow\AA$ be an exact symplectic twist diffeomorphism that preserves and twists the ends infinitely and let $p,q$ be two coprime integers. Then there exist at least two Birkhoff $(p,q)$-periodic orbits for $S$. 
\end{theorem}
\begin{remark}
The Theorem is proved via variational techniques. The periodic orbits correspond to critical points of an action defined in terms of the generating function. One of this point is a minimum and the other is a minimax if the critical points are isolated. 
  \end{remark}

In the analytic case, something more can be said on the topology of these orbits.

\begin{proposition}\label{unstable}
Let $\tilde{S}:\Sigma\rightarrow\RR^2$ be an exact symplectic twist embedding satisfying condition \eqref{def_cyl} and admitting a $(p,q)$-periodic orbit. Then there exists at least one $(p,q)$-periodic orbit that is unstable. 
\end{proposition}
\begin{proof}
  The proof is essentially given in \cite{maroTMNA,ortega_fp,ortega_book}. We give here a sketch. It is enough to prove that there exists at least one unstable fixed point of the area and orientation preserving 1-1 real analytic map $\sigma^{-p}\circ \tilde{S}^q$. Let us first note that since $\tilde{S}$ is twist, the map $\sigma^{-p}\circ \tilde{S}^q$ is not the identity. Actually, it is known (see for example \cite{herman})) that if $\tilde{S}$ is twist, then the image of a vertical line under $\tilde{S}^q$ is a positive path, i.e. a curve such that the angle between the tangent vector and the vertical is always positive. This implies that $\tilde{S}^q$ cannot be a horizontal translation.\\
By hypothesis the set of fixed points of $\sigma^{-p}\circ \tilde{S}^q$ is not empty so that applying \cite[Chapter 4.9, Theorem 15]{ortega_book} we deduce that every stable fixed point is an isolated fixed point.\\
Hence, if there exists some non isolated fixed point, it must be unstable. Finally, suppose that we only have isolated fixed points that are all stable. From \cite[Chapter 4.5, Theorem 12]{ortega_book} they must all have index $1$. On the other hand, the Euler characteristic of the cylinder is null, and by the Poincar\'e-Hopf index formula we have a contradiction. Hence, there must exists at least one fixed point that is unstable.
   \end{proof}

\begin{corollary}
  In the conditions of Theorem \ref{birk_orbits}, if $S$ is real analytic then there exists at least one unstable $(p,q)$-periodic orbit.
\end{corollary}

In the analytic case, the twist condition gives information on the structure of the set of $(p,q)$-periodic orbits. Actually the following result has been proved in \cite{maroTMNA,ortega_pb}

\begin{theorem}\label{maro_teo}
  Consider a $C^1$-embedding $\tilde{S}:\Sigma\rightarrow\RR$ satisfying property \eqref{def_cyl} and suppose it is exact symplectic and twist. Fix a positive integer $p$ and assume that for every $x\in\RR$ there exists $y\in (a,b)$ such that
  \begin{equation}\label{maro_cond}
\bar{x}(x,a)<x+p<\bar{x}(x,y).
  \end{equation}
Then the map $\sigma^{-p}\circ\tilde{S}$ has at least two fixed points in $[0,1)\times (a,b)$. Moreover, if $\tilde{S}$ is real analytic then the set of fixed points is finite or the graph of a real analytic $1$-periodic function. In the first case the index of such fixed points is either $-1$, $0$, $1$ and at least one is unstable. In the second case, all the fixed points are unstable.
  \end{theorem}

\begin{remark}
	Aubry-Mather theory gives a description, for fixed $p,q$ of those $(p,q)$-periodic orbits that are global minimizers. They can be ordered and if two of them are neighbouring, in the sense that there is no other minimal $(p,q)$-periodic orbit in the middle, then there are heteroclinic connections between them (\cite{bangert,katok_hass}). In this case, the $(p,q)$-periodic orbits are unstable.  On the other hand, they can form an invariant curve. We stress that in the analytic case Theorem \ref{maro_teo} gives the description of the set of all $(p,q)$-periodic orbits, not only those that are action minimizing.           
\end{remark}

\section{The bouncing ball map and its properties}\label{sec:tennis}

Consider the motion of a bouncing ball on a vertically
moving racket. We assume that the impacts do not affect the racket  whose vertical position is described by a
$1$-periodic $C^k$, $k\geq 2$, or real analytic function $f:\RR\rightarrow\RR$. Let us start getting the equations of motion, following \cite{maro5}. In an inertial frame, denote by $(\to,w)$ the time of impact and the corresponding velocity just after the bounce, and by
$(\Pto,\bar w)$ the corresponding values at the subsequent bounce.
From
the free falling condition we have
\begin{equation}\label{timeeq}
f(t) + w(\Pto-\to) - \frac{g}{2}(\Pto -\to)^2 = f(\Pto) \,,
\end{equation}
where $g$ stands for the standard acceleration due to gravity.
Noting that the velocity just before the impact at time $\Pto$ is $w-g(\Pto-\to)$, using
the elastic impact condition and recalling that the racket is not affected by the ball, we obtain
\begin{equation}
  \label{veleq}
\bar{w}+w-g(\Pto-\to) = 2\dot{f}(\Pto)\,,
\end{equation}
where $\dot{}$ stands for the derivative with respect to time. From conditions (\ref{timeeq}-\ref{veleq}) we can define a bouncing motion given an initial condition $(t,w)$ in the following way. If $w\leq\df(t)$ then we set $\bar{t}=t$ and $\bar{w}=w$. If $w>\df(t)$, we claim that we can choose $\bar{t}$ to be the smallest solution $\bar{t}> t$ of \eqref{timeeq}. Bolzano theorem gives the existence of a solution of \eqref{timeeq} considering
\[
F_t(\bar{t})=f(t)-f(\Pto) + w(\Pto-\to) - \frac{g}{2}(\Pto -\to)^2 
\]
and noting that $F_t(\bar{t})<0$ for $\Pto-\to$ large and $F_t(\bar{t})>0$ for $\Pto-\to\rightarrow 0^+$. Moreover, the infimum $\bar{t}$ of all these solutions is strictly larger than $t$ since if there exists a sequence $\bar{t}_n\rightarrow t$ satisfying \eqref{timeeq} then,
\[
 w - \frac{g}{2}(\Pto_n -\to) = (f(\Pto_n)- f(t)) /(\Pto_n-\to)
\]
that contradicts $w>\df(t)$ using the mean value theorem.
For this value of $\bar{t}$, condition \eqref{veleq} gives the updated velocity $\bar{w}$.

For $\Pto-\to>0$, we introduce the notation
\[
f[\to,\Pto]=\frac{f(\Pto)-f(\to)}{\Pto-\to},
\]
and write
\begin{equation}
  \label{w1}
\Pto = \to + \frac 2g w -\frac 2g f[\to,\Pto]\,,
\end{equation}
that also gives
\begin{equation}
  \label{w2}
\bar{w}= w -2f
[\to,\Pto]
+ 2\dot{f}(\Pto).
\end{equation}
Now we change to the moving frame attached to the racket, where the velocity after the impact is expressed as $v=w-\dot{f}(t)$,
and we get
the equations
\begin{equation}\label{eq:unb}
  \left\{
  \begin{split}
\Pto = {} & \to + \frac 2g \vo-\frac 2g f[\to,\Pto]+\frac 2g \df(\to)\textcolor{blue}{\,}
\\
\Pvo = {} & \vo - 2f[\to,\Pto] + \df (\Pto)+\df(\to)\textcolor{blue}{\,.}
\end{split}
\right.
\end{equation}
 By the periodicity of the
function $f$, the coordinate $t$ can be seen as an angle. Hence, equations \eqref{eq:unb} define formally a map
\[
\begin{array}{rcl}
\tilde\Psi: 
\RR & \longrightarrow & \RR \\
(\to,\vo) & \longmapsto & (\Pto, \Pvo),
\end{array}
\]
satisfying $\tilde\Psi\circ\sigma=\sigma\circ\tilde\Psi$ and the associated map of the cylinder $\Psi:\AA\rightarrow\AA$.

This is the formulation considered by Kunze and Ortega
\cite{kunzeortega2}. Another approach was considered by Pustylnikov in
\cite{pust} and leads to a map that is equivalent to
\eqref{eq:unb}, see \cite{maro3}.
Noting that $w>\df(t)$ if and only if $v>0$, we can define a bouncing motion as before and denote it as a sequence $(t_n,v_n)_{n\in\ZZ^+}$ with $\ZZ^+=\{n\in\ZZ \::\: n\geq 0\}$ such that $(t_n,v_n)\in \TT\times [0,+\infty)$ for every $n\in\ZZ^+$.

The maps $\Psi$ and its lift $\tilde{\Psi}$ are only defined formally. In the following lemma we state that they are well defined and have some regularity.  
Let us introduce the notation $\RR_{v_*}=\{v\in\RR \: :\: v>v_* \}$,  $\AA_{v_*} = \TT\times \RR_{v_*}$ and $\RR^2_{v_*}=\RR\times\RR_{v_*}$. We will denote the $\sup$ norm by $\norm{\cdot}$ and recall that $f\in C^k(\TT), k\geq 2$ or real analytic.

\begin{lemma}
  \label{well_def}
There exists $v_*>4\norm{\df}$ such that the map $\Psi:\AA_{v_*}\rightarrow \AA$ is a $C^{k-1}$ embedding. If $f$ is real analytic, then $\Psi$ is a real analytic embedding. 
\end{lemma}
%}
\begin{proof}
The proof is essentially given in \cite{maro5}. We give here a sketch.
To prove that the map is well defined and $C^{k-1}$ we denote $v_{**}=4\norm{\df}$ and apply the implicit function theorem to the $C^{k-1}$ function $F :\{(\to,\vo,\Pto,\Pvo)\in \AA_{v_{**}} \times \RR^2 \: :\:\to\neq\Pto   \} \rightarrow \RR^2$
given by

\begin{equation*}
%\textcolor{red}{
% \RR_{v_*} \times (\TT\times\RR) \ni}
 F(\to,\vo,\Pto,\Pvo):=
\left(
\begin{split}
& \Pto - \to - \frac 2g \vo + \frac 2g f[\to,\Pto]-\frac 2g \df (\to) \\
& \Pvo - \vo + 2f[\to,\Pto] - \df(\Pto)-\df(\to)
\end{split}
\right),
\end{equation*}
This gives the existence of a $C^{k-1}$ map $\Psi$ defined in $\AA_{v_{**}}$ such that $F(\to,\vo,\Psi(\to,\vo) )=0$. If $f$ is real analytic, we get that $\Psi$ is real analytic applying the analytic version of the implicit function theorem. \\
One can easily check that $\Psi$ is a local diffeomorphism since
\[
\det \Dif_{\to,\vo} \Psi (t,v) =-\frac{\det (\Dif_{\to,\vo} F(\to,\vo,\Psi(\to,\vo) ))}{ \det (\Dif_{\Pto,\Pvo} F(\to,\vo,\Psi(\to,\vo) ))} \neq 0 \quad\mbox{on }\AA_{v_{**}}.
\]
To prove that $\Psi$ is a global embedding we need to prove that it is injective in $\AA_{v_*}$ for $v_*$ eventually larger than $v_{**}$. This can be done as in \cite{maro5}. 
\end{proof}

\begin{remark}
  Note that we cannot guarantee that if $(\to_0,\vo_0) \in \AA_{v_*}$ then $\Psi(\to_0,\vo_0) \in \AA_{v_*}$. This is reasonable, since the ball can slow down decreasing its velocity at every bounce.
  %$Hence, in general, given $(\to_0,\vo_0) \in \AA_{v_*}$ we can only define a finite number $n_{max}$ of iterates of the map by the relation: 
 % \[
 % \Psi^{n_{max}} (\to_0,\vo_0) \in \AA_{v_*}, \quad\mbox{ but }\quad   \Psi^{n_{max}+1} (\to_0,\vo_0) \notin \AA_{v_*}
 % \]
  %and the same for backwards iterates.
  However, a bouncing motion is defined for $v\geq 0$.% defined in this section we will consider only orbits contained in $\AA_{v_*}$.
\end{remark}

\begin{remark}
  From the physical point of view, the condition $\Psi^n(\to_0,\vo_0)\in\AA_{v_*}$ for every $n$, implies that we can only hit the ball when it is falling. To prove it, suppose that $\to_0 =0$ and let us see what happens at the first iterate. The time at which the ball reaches its maximum height is $t^{max}=\frac{\vo_0}{g}$. On the other hand, the first impact time $\Pto$ satisfies,
  \[
\Pto \geq \frac{2}{g}\vo_0 - \frac{4}{g}\norm{\df}= t^{max}\left(         2-\frac{4}{\vo_0}\norm{\df} \right) > t^{max},
\]
where the last inequality comes from $\vo_0\in\RR_{v_*}$ and $v_*>4\norm{\df}$.
\end{remark}

The map $\tilde\Psi$ is exact symplectic if we pass to the variables time-energy $(\to,\Eo)$ defined by
\[
 (\to,\Eo) = \left(\to,\frac{1}{2}\vo^2\right),
\]
obtaining the conjugated map
\[
\Phi :\AA_{\Eo_*}
\longrightarrow  \AA, \qquad \Eo_*=\frac{1}{2}v_*^2
%\qquad \mapPP(\to,\Eo) = \Psi \circ
%\mapP \circ \Psi^{-1} (\to,\Eo)
\]
defined by
\begin{equation}\label{eq:unbe}
  \left\{
  \begin{split}
\Pto =  & \to + \frac 2g \sqrt{2\Eo}-\frac 2g f[\to,\Pto]+\frac 2g \df(\to)
\\
\PEo =  & \frac{1}{2}\left( \sqrt{2\Eo} - 2f[\to,\Pto] + \df (\Pto)+\df(\to) \right)^2,
\end{split}
\right.
\end{equation}
that by Lemma \ref{well_def} is a $C^{k-1}$-embedding or real analytic if $f$ is real analytic. More precisely, we have the following
\begin{lemma}
  \label{lemma:exact}
  The map $\Phi$ is exact symplectic and twist in $\AA_{e_*}$.
 % The generating function is given by
 % \begin{equation}
 %   \label{genfun}
 %   \begin{split}
 %     h(\to,\Pto)= &  \frac{g^2}{24}(\Pto-\to)^3+\frac{g}{2}(f(\Pto)+f(\to))(\Pto-\to)-\frac{(f(\Pto)-f(\to))^2}{2(\Pto-\to)}\\
 %     & -g\int^{\Pto}_{\to}f(s)ds+\frac{1}{2}\int^{\Pto}_{\to}
 %     (\dot{f}(s))^2ds,
 %   \end{split}
 % \end{equation}
 % where
 % \[
 % \Pto-\to\in\Omega =\{ (\to,\Pto)\in\RR^2 : \Pto>T(\to)\}
 % \]
 % for a $C^{k-1}$ function $T:\RR\rightarrow \RR$ such that $T(t+1)=T(t)+1$ and $\frac{2}{g}(v_*-2\norm{\df})<T(t)-t<\frac{2}{g}(v_*+2\norm{\df})$.\\
  Moreover, $\Phi$ preserves and twists infinitely the upper end.
  
\end{lemma}

The map $\Phi$ is not defined in the whole cylinder. However, it is possible to extend to the whole cylinder preserving its properties. More precisely:
\begin{lemma}
  \label{extension}
There exists a $C^{k-2}$ exact symplectic and twist diffeomorphism $\bar{\Phi}:\AA\rightarrow\AA$ such that $\bar{\Phi} \equiv \Phi$ on $\AA_{e_*}$ and $\bar{\Phi}\equiv\Phi_0$ on $\AA\setminus\AA_{\frac{e_*}{2}}$ where $\Phi_0$ is the integrable twist map $\Phi_0(t,e)=(t+e,e)$. Moreover, $\bar{\Phi}$ preserves the ends of the cylinder and twists them infinitely. If $f$ is real analytic, then the extension $\bar{\Phi}$ is $C^\infty$.
\end{lemma}

Due to Lemma \ref{well_def} and the fact that the maps $\Phi$ and $\Psi$ are conjugated we can consider the lift $\tilde{\Phi}:\RR^2_{e_*}\rightarrow\RR^2$ and give the following
\begin{definition}
  A complete bouncing motion $(t_n,e_n)_{n\in\ZZ}$ is a complete orbit of the map $\tilde{\Phi}$.
  \end{definition}

In the following section we will study the existence and properties of periodic complete bouncing motions as orbits of the map $\tilde{\Phi}$ defined in \eqref{eq:unbe}. 

%\begin{remark}\label{rem:isotop}
%From the topological point of view, the map $\Phi$ and $\bar{\Phi}$ are area and orientation preserving, being symplectic. Moreover, the map $\bar{\Phi}$ is isotopic to the identity, in the sense that there exists a function. This comes from the fact that there are only two classes of isotopy for orientation preserving diffeomophisms of the cylinder, corresponding to the identity $(t,e)\mapsto (t,e)$ and minus the identity $(t,e)\mapsto (-t,-e)$ (see CITE[]). Since $\bar{\Phi}$ preserves the ends of the cylinder it must be isotopic to the identity. Hence, the map $\Phi:\AA_{e_*}\rightarrow\AA$ is isotopic to the inclusion, being the isotopy the restriction of the previous isotopy to $\AA_{e_*}$.  
%  \end{remark}

\section{Periodic bouncing motions}\label{sec:per}
The existence of periodic complete bouncing motions follows from an application of Aubry-Mather theory. In this section we prove it and in the analytic case we give some results on the structure of such motions and their stability.

We start saying that a complete bouncing motion is $(q,p)$-periodic if in time $p$ the ball makes $q$ bouncing before repeating the motion, more precisely:
\begin{definition}
  Given two coprime integers $q,p\in\ZZ^+$, a complete bouncing motion $(t_n,e_n)_{n\in\ZZ}$ is $(q,p)$-periodic if the corresponding orbit of $\tilde{\Phi}$ is $(p,q)$-periodic. Moreover, we say that it is stable if the corresponding orbit is stable.
\end{definition}

The existence of two $(q,p)$-periodic complete bouncing motions comes from an application of Theorem \ref{birk_orbits}.

\begin{theorem}\label{pre_bouncing}
For every $f\in C^3$ there exists $\alpha>0$ such that for every positive coprime $p,q$ satisfying $p/q>\alpha$ there exist two different $(q,p)$-periodic complete bouncing motions. Moreover, if $f$ is real analytic, then at least one of the $(q,p)$-periodic complete bouncing motion is unstable.  
  \end{theorem}
\begin{proof}
  By Lemma \ref{lemma:exact}, the map $\Phi$ defined in \eqref{eq:unbe} is a $C^2$ exact symplectic twist embedding in $\AA_{e_*}$ for some large $e_*$ depending on $\norm{\df}$. Moreover $\Phi$ preserves and twists infinitely the upper end. Its extension $\bar{\Phi}$ coming from Lemma \ref{extension} satisfies the hypothesis of Theorem \ref{birk_orbits} and admits for every coprime $p,q$ two Birkhoff $(q,p)$-periodic orbit. Consider the Birkhoff $(q,p)$-periodic orbits for $p,q$ positive and $p/q$ large enough such that
  \begin{equation}\label{choice_pq}
  \frac{p}{q}-1-\frac{4}{g}\norm\df>\frac{2}{g}\sqrt{2e_*}.
  \end{equation}
  Since Birkhoff periodic orbits are cyclically ordered, from \cite[lemma 9.1]{gole} we have that they satisfy the estimate $t_{n+1}-t_n>p/q - 1$. On the other hand, from \eqref{eq:unbe},
  \[
  t_{n+1}-t_n\leq\frac{2}{g}\sqrt{2e_n}+\frac{4}{g}\norm\df 
  \]
  so that it must be
  \[
\frac{2}{g}\sqrt{2e_n}+\frac{4}{g}\norm\df>\frac{p}{q}-1.
  \]
By the choice of $p/q$ in \eqref{choice_pq} we have that $e_n>e_*$ for every $n\in\ZZ$ so that these Birkhoff periodic orbits are all contained in $\AA_{e_*}$ and so they are orbits of the original map $\Phi$. If $f$ is real analytic the result on the instability follows from Proposition \ref{unstable}.  
\end{proof}

%In view of remark \ref{rem_periodic} we would like to apply Theorem \ref{franks_teo} to the map $\sigma^{-p}\circ\Phi^q$. Let us first prove that

Theorem \ref{pre_bouncing} gives the existence of $(p,q)$-periodic bouncing motions but does not give information on the topological structure of the set of $(p,q)$-periodic bouncing motions for fixed values of $(p,q)$. This is a complicated issue and some results comes from Aubry-Mather theory. However, here we will see which results can be obtained using Theorem \ref{maro_teo}. To state this result we give the following
\begin{definition}
  %We say that a $(p,q)$-periodic complete bouncing motion $(t_n,v_n)_{n\in\ZZ}$ is isolated if there exists $\delta>0$ such that every other complete bouncing motion $(\hat{t}_n,\hat{v}_n)_{n\in\ZZ}$ such that $|\hat{t}_0-t_0|+|\hat{v}_0-v_0|<\delta$ is not $(p,q)$-periodic. On the other hand
  We say that the set of $(p,q)$-periodic complete bouncing motions is (analytically) degenerate if there exists a real analytic curve $(t(s),e(s))$ such that $(t(s+1),e(s+1))=(t(s)+1,e(s))$ for every $s\in\RR$, the function $t(s)$ is bijective for $s\in[0,1)$ and $(t_n,e_n)_{n\in\ZZ}$ is a $(p,q)$-periodic complete bouncing motion if and only if there exist $n_0,s_0$ such that $(t_{n_0},e_{n_0})=(t(s_0),e(s_0))$.          
  \end{definition}

The following result is a quite direct consequence of Lemma \ref{lemma:exact}.
\begin{proposition}
If $f$ is real analytic, then there exists $\alpha>0$ such that for every $p>\alpha$ the set of $(p,1)$-periodic complete bouncing motions is either finite or degenerate. In the first case at least one $(p,1)$-periodic complete bouncing motion is unstable. In the degenerate case, all $(p,1)$-periodic complete bouncing motion are unstable.        
  \end{proposition}
\begin{proof}
  By Lemma \ref{lemma:exact} the map $\tilde{\Phi}$ is exact symplectic and twist on $\RR^2_{e_*}$. Moreover, let us choose $a$ such that $e_*<\sqrt{2a}$ and $p$ such that
  \[
\frac{gp}{2}-2\norm{\dot{f}}>\sqrt{2a}.
\]
Let us start with the following estimates for the lift $\tilde{\Phi}$ that can be easily proved by induction on $n$ from \eqref{eq:unbe}:
  \begin{equation}\label{stim_t_e}
  	|\sqrt{2e_n}-\sqrt{2e}|\leq 4n\norm{\dot{f}}, \qquad \left|t_n-t-\frac{2}{g}n\sqrt{2e}\right|\leq 4n^2\frac{\norm{\dot{f}}}{g}.  
  \end{equation}
  These give
  \[
\bar{t}(t,a)-t\leq \frac{2}{g}\sqrt{2a}+4\frac{\norm{\dot{f}}}{g}<p.
\]
On the other hand, Lemma \ref{lemma:exact} also gives that $\tilde{\Phi}$ twists the upper end infinitely, i.e. $\lim_{e\rightarrow +\infty}\bar{t}(t,e)-t=+\infty$ uniformly in $t$. Hence, condition \eqref{maro_cond} is satisfied in the strip $\Sigma=[a,+\infty)$ for every $p$. The conclusion comes from the application of Theorem \ref{maro_teo} and the fact that $(p,1)$-periodic complete bouncing motions corresponds to the fixed points of the map $\sigma^{-p}\circ\tilde{\Phi}$.
  \end{proof}

This result is not trivially extended to $(p,q)$-periodic motions for $q\geq 1$ since $\Phi^q$ need to be exact symplectic and twist. The twist condition is in general not preserved by composition, while the exactness is, as shown in the following result, inspired by \cite{bosc_ort}
\begin{lemma}\label{lemma_isot_exact}
For every $q>0$ there exists $e_\#\geq e_*$ such that for every $p>0$ the map $\sigma^{-p}\circ\tilde{\Phi}^q:\RR^2_{e_\#}\rightarrow\RR^2$ is exact symplectic.% and isotopic to the inclusion.  
  \end{lemma}
\begin{proof}
  Since the map $\tilde{\Phi}$ is defined in $\RR^2_{e_*}$, the image $\tilde{\Phi}(\RR^2_{e_*})$ could not be contained in $\RR^2_{e_*}$ so that the iterate could not be defined. From \eqref{eq:unb} we have that $|\bar{v}-v|\leq 4\norm{\dot{f}}$ from which, the map $\tilde{\Psi}^q$ is well defined in $\RR^2_{v_\#}$ with $v_\#=v_* + 4 q \norm{\dot{f}}$. Hence passing to the variables $(t,e)$, the map $\tilde{\Phi}^q$ is defined in $\RR^2_{e_\#}$ with $e_\#=\frac{1}{2}v_\#^2$.
  %Moreover, the map $\sigma^{-p}$ is trivially isotopic to the identity and the composition of maps isotopic to the identity is still isotopic to the identity. Hence, as done in remark \ref{rem:isotop} the extended map $\sigma^{-p}\circ\tilde{\Phi}^q$ is isotopic to the identity and its restriction to $\AA_{e_\#}$ is isotopic to the inclusion. Finally, the map $\sigma^{-p}\circ\Phi^q$ is also exact symplectic. Actually,
  Since $\Phi$ is exact symplectic, there exists $V:\RR^2_{e_\#}\rightarrow\RR$ such that, defining $\lambda=e dt$ we have $\tilde{\Phi}^*\lambda-\lambda=dV$. Hence, denoting $V_1=V+V\circ\tilde{\Phi}+\dots+V\circ\tilde{\Phi}^{q-1}$ it holds that $V_1\circ\sigma=V_1$ on $\RR^2_{e_\#}$  and
  \begin{align*}
    dV_1 &= dV+\tilde{\Phi}^*dV+\dots+(\tilde{\Phi}^{q-1})^*dV \\
    & = \tilde{\Phi}^*\lambda-\lambda + (\tilde{\Phi}^2)^*\lambda-\tilde{\Phi}^*\lambda             +\dots+(\tilde{\Phi}^q)^*\lambda-(\tilde{\Phi}^{q-1})^*\lambda \\
    &= (\tilde{\Phi}^q)^*\lambda-\lambda
    \end{align*}
  from which $\tilde{\Phi}$ is exact symplectic. Finally, we conclude noting that by the definition of $\sigma^{-p}$,  $(\sigma^{-p}\circ\tilde{\Phi}^q)^*\lambda = (\tilde{\Phi}^q)^*\lambda$.  
\end{proof}

Concerning the twist condition the following technical result holds.
\begin{lemma}\label{twist_q}
  Let $f$ be $C^2$. For every $q\geq 1$ there exist $\epsilon_q>0$, $e^q>e_\#$, such that if $\norm{\ddf}<\epsilon_q$ then
  \[
\frac{\partial t_q}{\partial e}=\frac{2q}{g\sqrt{2e}}(1+\tilde{f}_q(t,e))
  \]
where $|\tilde{f}_q(t,e)|< 1/2$ on $\RR^2_{e^q}$.
  \end{lemma}

\begin{proof}
  To simplify the computation, let us perform the change of variables $y=\sqrt{2e}+\df(t)$ so that \eqref{eq:unbe} becomes
\begin{equation}\label{eq:unby}
  \left\{
  \begin{split}
\Pto =  & \to + \frac 2g y-\frac 2g f[\to,\Pto]
\\
\bar{y} =  & y  - 2f[\to,\Pto] + 2\df (\Pto).
\end{split}
\right.
\end{equation}
Since $\partial t_q/\partial e=(\partial t_q/\partial y) (\partial y/\partial e)$ it is enough to prove that for every $q\geq 1$ there exist $\epsilon_q>0$ and $y^q$ large, such that if $\norm{\ddf}<\epsilon_q$ then 
  \begin{equation}\label{new_th}
\frac{\partial t_q}{\partial y}=\frac{2q}{g}(1+\tilde{f}_q(t,y))
\end{equation}
where $|\tilde{f}_q(t,y)|< 1/2$ on $\RR^2_{y^q}$.

Let us start with some estimates that hold for every $q\geq 1$. It comes from \eqref{eq:unby} that 
\begin{equation}\label{y_q}
y_q=y+2\sum_{i=1}^{q}\df(t_i)+2\sum_{i=1}^{q}f[t_{i-1},t_{i}]
\end{equation}
so that
\[
|y_q-y|\leq 4q\norm\df.
\]
Using it,
\begin{equation}\label{t_q}
|t_q-t_{q-1}|\geq \frac{2}{g}y-\frac{2}{g}(4q+1)\norm\df,
\end{equation}
from which there exists $y^q$ large enough and $C_q$ such that 
\begin{equation}\label{fqq-11}
\left|\partial_{t_q}f[t_{q-1},t_q]\right|=\left|\frac{\dot{f}(t_q)-f[t_{q-1},t_q]}{t_q-t_{q-1}}\right|\leq
\frac{g\norm{\dot{f}}}{y-(4q+2)\norm{\dot{f}}}<\frac{C_q}{y} \qquad \mbox{on } \RR^2_{y^q}
\end{equation}
and analogously
\begin{equation}\label{fqq-12}
\left|\partial_{t_{q-1}}f[t_{q-1},t_q]\right|<\frac{C_{q-1}}{y} \qquad \mbox{on } \RR^2_{y^{q-1}}
\end{equation}
  Now we can start the proof by induction on $q\geq 1$. For $q=1$ we have, differentiating the first equation in \eqref{eq:unby}
  \begin{equation}\label{t_1_y}
\frac{\partial t_1}{\partial y}\left(1+\frac{2}{g}\partial_{t_{1}}f[t,t_1]\right)=\frac{2}{g}y
\end{equation}
  from which, using \eqref{fqq-11} we get the initial step taking a suitably larger value of $y^1$.\\
For the inductive step, let us suppose \eqref{new_th} to be true for $i=1,\dots,q-1$. By implicit differentiation

\begin{equation}\label{t_q_e}
\frac{\partial t_q}{\partial y}\left(1+\frac{2}{g}\partial_{t_{q}}f[t_{q-1},t_q]\right)=\frac{2}{g}\frac{\partial y_{q-1}}{\partial y}+\frac{\partial t_{q-1}}{\partial y}\left(1-\frac{2}{g}\partial_{t_{q-1}}f[t_{q-1},t_q]\right)
\end{equation}
From \eqref{y_q} and the inductive hypothesis we have
\begin{align*}
  \frac{\partial y_{q-1}}{\partial y}&=1+2\sum_{i=1}^{q-1}\left(\ddf(t_i) \frac{\partial t_i}{\partial y}-\partial_{t_{i}}f[t_{i-1},t_{i}]\frac{\partial t_i}{\partial y} -\partial_{t_{i-1}}f[t_{i-1},t_{i}]\frac{\partial t_{i-1}}{\partial y}\right)\\
   &=1+2\sum_{i=1}^{q-1}\left(\ddf(t_i) (1+\tilde{f}_i)-\partial_{t_{i}}f[t_{i-1},t_{i}](1+\tilde{f}_i) -\partial_{t_{i-1}}f[t_{i-1},t_{i}](1+\tilde{f}_{i-1})\right).
  \end{align*}
Since by (\ref{fqq-11}-\ref{fqq-12}) for every $i$ $|\partial_{t_{i}}f[t_{i-1},t_{i}]|$ tends to zero uniformly as $y \rightarrow +\infty$ and $|\tilde{f}_i|<1/2$ for $y$ large, we can find new constants $C_{q-1}$ and $y^{q-1}$ such that on $\RR^2_{y^{q-1}}$,
\begin{equation}\label{stim_y}
\frac{\partial y_{q-1}}{\partial y} = 1 + \bar{f}_{q-1}(t,y)  \qquad \mbox{with } |\bar{f}_{q-1}|\leq C_{q-1}\norm\ddf.
\end{equation}
Using it and the inductive hypothesis in \eqref{t_q_e} we get
\begin{align}\label{final}
  \frac{\partial t_q}{\partial y}\left(1+\frac{2}{g}\partial_{t_{q}}f[t_{q-1},t_q]\right)&=\frac{2}{g}(1 + \bar{f}_{q-1}(t,y))+\frac{2(q-1)}{g}(1+\tilde{f}_{q-1}(t,y))\left(1-\frac{2}{g}\partial_{t_{q-1}}f[t_{q-1},t_q]\right)\\
  &=\frac{2q}{g}\left(1+ \tilde{f}_q(t,y)\right)
\end{align}
where,
\[
\tilde{f}_q(t,y)=\frac{1}{q}\bar{f}_{q-1}(t,y)+\frac{q-1}{q}\tilde{f}_{q-1}(t,y)+\frac{2(q-1)}{gq}\partial_{t_{q-1}}f[t_{q-1},t_q](1+\tilde{f}_{q-1}(t,y)).
\]
Now (\ref{fqq-11}-\ref{fqq-12}) and \eqref{stim_y} imply 
\[
|\tilde{f}_q|<\frac{C_{q-1}}{q}\norm\ddf+\frac{q-1}{2q}+\frac{C'_{q-1}}{y}
\]
so that we can find $\epsilon_q$ and $y_q$ such that if $\norm\ddf<\epsilon_q$ then $|\tilde{f}_q|<\frac{1}{2}-\frac{1}{2q}$ on $\RR^2_{y^{q}}$. Plugging it into \eqref{final} and using again (\ref{fqq-11}-\ref{fqq-12}) we get the thesis, eventually increasing $y^q$.

\end{proof}

This is used to prove the following result
\begin{proposition}
Suppose that $f$ is real analytic. For every $q>0$ there exist $\alpha>0$ and $\epsilon_q>0$ such that if $p>\alpha$ and $\norm{\ddf}<\epsilon_q$ then the set of $(p,q)$-periodic complete bouncing motion is either finite or degenerate. In the first case at least one $(p,q)$-periodic complete bouncing motion is unstable. In the degenerate case, all $(p,q)$-periodic complete bouncing motion are unstable.  
\end{proposition}

\begin{proof}
  We would like to apply Theorem \ref{maro_teo} to the map $\tilde{\Phi}^q$, noting that $(p,q)$-periodic bouncing motions correspond to fixed points of the map $\sigma^{-p}\circ\tilde{\Phi}^q$. Let us fix $q>0$. In Lemma \ref{lemma_isot_exact} we proved that $\tilde{\Phi}^q$ is exact symplectic in $\RR^2_{e_\#}$ for some $e_\#$ depending on $q$. Moreover, by Lemma \ref{twist_q}, there exist $\epsilon_q$ and $e^q>e_\#$ such that if $\norm{\ddf}<\epsilon_q$ then $\tilde{\Phi}^q$ is also twist on $\RR^2_{e^q}$. Now choose $p>0$ such that
  \[
  \frac{gp}{2q}-2q\norm{\dot{f}}>e^q.
  \]  
   Hence, there exist $b>a$ such that
  \[
  e^q<\sqrt{2a}<\frac{gp}{2q}-2q\norm{\dot{f}}<\frac{gp}{2q}+2q\norm{\dot{f}}<\sqrt{2b}.
  \] 
  This choice for $a,b$ gives condition \eqref{maro_cond} on the strip $\Sigma=\RR\times (a,b)$. Actually, from \eqref{stim_t_e},
  \begin{align*}
  t_q(t,b)-t & \geq\frac{2}{g}q\sqrt{2b}-4q^2\frac{\norm{\dot{f}}}{g}>p \\
  t_q(t,a)-t & \leq\frac{2}{g}q\sqrt{2a}+4q^2\frac{\norm{\dot{f}}}{g}<p.
  \end{align*}
  This concludes the proof.
  \end{proof} 

%\section*{Acknowledgements}
%The author would like to thank the unknown referee for several valuable advice that significantly improved the final version of the paper.

\end{document}